\documentclass[a4paper,12pt]{article}

\usepackage[english]{babel}
\usepackage{tikz}
\usetikzlibrary{matrix,arrows,decorations.markings}
\usepackage{amsmath,amsfonts,amssymb,amsthm,url,textcomp}
\usepackage{csquotes}
\usepackage{a4wide}
\usepackage[numbers]{natbib}
\usepackage{fancyhdr}
\usepackage{anyfontsize}
\usepackage{hyperref}
\usepackage{hhline}
\usepackage{leftidx}
\usepackage[ruled,vlined]{algorithm2e}

\allowdisplaybreaks

\newtheorem{theorem}{Theorem}\numberwithin{theorem}{section}

\newtheorem{lemma}[theorem]{Lemma}
\newtheorem{corollary}[theorem]{Corollary}

\newtheorem*{proposition*}{Proposition}
\newtheorem{notation}[theorem]{Notation}

\numberwithin{theoremm}{subsection}

\numberwithin{theoremmm}{subsubsection}

\theoremstyle{remark}

\newcommand{\C}{\operatorname{C}}

\newcommand{\F}{\operatorname{F}}

\renewcommand{\F}{\operatorname{F}}

\newcommand{\IZ}{\mathbb{Z}}

\begin{document}

\title{Approximability of word maps by homomorphisms}

\author{Alexander Bors\thanks{University of Salzburg, Mathematics Department, Hellbrunner Stra{\ss}e 34, 5020 Salzburg, Austria. \newline E-mail: \href{mailto:alexander.bors@sbg.ac.at}{alexander.bors@sbg.ac.at} \newline The author is supported by the Austrian Science Fund (FWF):
Project F5504-N26, which is a part of the Special Research Program \enquote{Quasi-Monte Carlo Methods: Theory and Applications}. \newline 2010 \emph{Mathematics Subject Classification}: 20D60. \newline \emph{Key words and phrases:} Finite groups, Word maps, Word equations.}}

\date{\today}

\maketitle

\abstract{Generalizing a recent result of Mann, we show that there is an explicit function $f:\left(0,1\right]\rightarrow\left(0,1\right]$ such that for every reduced word $w$, say in $d$ variables, there is an explicit reduced word $v$ in at most $3d$ variables (nontrivial if the length of $w$ is at least $2$) such that for all $\rho\in\left(0,1\right]$, the following holds: If $G$ is any finite group for which the word map $w_G:G^d\rightarrow G$ agrees with some fixed homomorphism $G^d\rightarrow G$ on at least $\rho|G|^d$ many arguments, then the word map $v_G:G^{3d}\rightarrow G$ has a fiber of size at least $f(\rho)|G|^{3d}$. We also discuss some applications of this result.}

\section{Introduction}\label{sec1}

Various authors have studied finite groups $G$ with an automorphism mapping a certain minimum proportion of elements of $G$ to their $e$-th power, for $e\in\{-1,2,3\}$ fixed. Some notable results in this context are the following: If $G$ has an automorphism inverting (resp.~squaring, cubing) more than $\frac{3}{4}|G|$ (resp.~$\frac{1}{2}|G|$, $\frac{3}{4}|G|$) many elements of $G$, then $G$ is abelian, see \cite{GPa} (resp.~\cite[Theorem 3.5]{Lie73a}, \cite[Theorem 4.1]{Mac75a}), and if $G$ has an automorphism inverting (resp.~squaring, cubing) more than $\frac{4}{15}|G|$ (resp.~$\frac{7}{60}|G|$, $\frac{4}{15}|G|$) many elements of $G$, then $G$ is solvable, see \cite[Corollary 3.2]{Pot88a} (resp.~\cite[Theorem C]{Heg03a}, \cite[Theorem 4.1]{Heg09a}). In \cite{Bor16a}, the author studied finite groups with an automorphism inverting, squaring or cubing at least $\rho|G|$ many elements, for a fixed, but arbitrary $\rho\in\left(0,1\right]$, and recently, in \cite{Man17a}, Mann proved the following, yielding an approach that works under the weaker assumption where the word \enquote{automorphism} is replaced by \enquote{endomorphism} and consists in rewriting the assumption into a nontrivial lower bound on the proportion of solutions to a certain word equation in three variables over $G$:

\begin{theorem}\label{mannTheo}(Mann, \cite[Theorem 9]{Man17a})
There is a function $f:\left(0,1\right]\rightarrow\left(0,1\right]$ such that for all $e\in\IZ$, all $\rho\in\left(0,1\right]$ and all finite groups $G$: If $G$ has an endomorphism $\varphi$ such that $\varphi(x)=x^e$ for at least $\rho|G|$ many $x\in G$, then the word equation $(xyz)^e=x^ey^ez^e$ has at least $f(\rho)|G|^3$ many solutions over $G$.\qed
\end{theorem}

The aim of this note is to generalize Mann's approach, allowing for the use of results on word equations to study a larger class of problems; we will also give an exemplary application of this more general method. Recall that to every (reduced) word $w\in\F(X_1,\ldots,X_d)$ and every group $G$, there is associated a \emph{word map} $w_G:G^d\rightarrow G$, induced by substitution.

\begin{theorem}\label{mainTheo}
There is a function $f:\left(0,1\right]\rightarrow\left(0,1\right]$ such that for all words $w\in\F(X_1,\ldots,X_d)$, all $\rho\in\left(0,1\right]$ and all finite groups $G$: If there is a homomorphism $\varphi:G^d\rightarrow G$ such that $|\{\vec{g}\in G^d\mid\varphi(\vec{g})=w(\vec{g})\}|\geq\rho|G|^d$, then the following word equation in $3d$ pairwise distinct variables $x_i,y_i,z_i$, $i=1,\ldots,d$, has at least $f(\rho)|G|^{3d}$ many solutions over $G$:

\[
w(x_1^{-1}y_1z_1,\ldots,x_d^{-1}y_dz_d)=w(x_1,\ldots,x_d)^{-1}w(y_1,\ldots,y_d)w(z_1,\ldots,z_d).
\]
\end{theorem}

Mann's Theorem \ref{mannTheo} is the special case where $w=X_1^e$ for some $e\in\IZ$ (note that the inversion from Theorem \ref{mainTheo} can be removed in that special case through the substitution $x_1\mapsto x_1^{-1}$). Using Theorem \ref{mainTheo} and recent results on fibers of word maps (see \cite{LS17a}), one gets:

\begin{corollary}\label{mainCor}
Let $w\in\F(X_1,\ldots,X_d)$ be a reduced word of length at least $2$. Then for all $\rho\in\left(0,1\right]$, the orders of the nonabelian composition factors of a finite group $G$ such that $w_G$ agrees with some homomorphism $G^d\rightarrow G$ on at least $\rho|G|^d$ many arguments are bounded in terms of $w$ and $\rho$.
\end{corollary}

\section{Proof of Theorem \ref{mainTheo}}\label{sec2}

Our proof of Theorem \ref{mainTheo} actually yields an explicit example of such a function $f$: One can take $f:=f_1\cdot f_2$, where $f_1,f_2:\left(0,1\right]\rightarrow\left(0,1\right]$ are given by Notation \ref{fNot} below. We note that in principle, Mann's counting argument in \cite[proof of Theorem 9]{Man17a} gives an explicit example of a function $f$ as in Theorem \ref{mannTheo} and that essentially the same argument can be applied to give an explicit proof of Theorem \ref{mainTheo}. Still, we will give a slightly different argument, for two reasons: Firstly, Mann's argument simplifies the situation a bit by replacing two quantities by (asymptotic) approximations, and to get a correct value for $f(\rho)$ (not just an approximation to a correct value) one would have to go through some, as Mann says himself, \enquote{tedious} computations. Secondly, our alternative approach allows us to highlight Lemma \ref{setLem} below, which we think is, in spite of its elementarity, an interesting result in its own right.

Let us first introduce the two functions $f_1$ and $f_2$ that make $f$ from Theorem \ref{mainTheo} explicit:

\begin{notation}\label{fNot}
We introduce the following functions $f_1,f_2:\left(0,1\right]\rightarrow\left(0,1\right]$:

\begin{enumerate}
\item $f_1(\rho):=\min\{\rho^2/(12\lceil 2\rho^{-1}\rceil),\rho^3/(4\lceil 2\rho^{-1}\rceil)\}$, for each $\rho\in\left(0,1\right]$.
\item $f_2(\rho):=\rho/(\lceil 2\rho^{-1}\rceil\cdot(\lceil 2\rho^{-1}\rceil+1))$, for each $\rho\in\left(0,1\right]$.
\end{enumerate}
\end{notation}

Our proof of Theorem \ref{mainTheo} relies on the following combinatorial lemma, already mentioned above and an extension of \cite[Lemma 2.1.2]{Bor16a}:

\begin{lemma}\label{setLem}
For all $\rho\in\left(0,1\right]$, all finite sets $X$ and all families $(M_i)_{i\in I}$ of subsets of $X$ such that $|M_i|\geq\rho|X|$ for all $i\in I$ and $|I|\geq\rho|X|$, there are at least $f_1(\rho)|X|^2$ many pairs $(i_1,i_2)\in I^2$ such that $|M_{i_1}\cap M_{i_2}|\geq f_2(\rho)|X|$.
\end{lemma}

\begin{proof}
We make a case distinction.

\begin{itemize}
\item Case 1: $|X|<4\lceil 2\rho^{-1}\rceil\rho^{-1}$. Then, using that for all $i\in I$,

\[
|M_i\cap M_i|=|M_i|\geq\rho|X|\geq f_2(\rho)|X|,
\]

we see that for at least

\[
|I|=\frac{1}{|I|}{|I|^2}\geq\frac{1}{4\lceil 2\rho^{-1}\rceil\rho^{-1}}\cdot\rho^2|X|^2\geq f_1(\rho)|X|^2
\]

many pairs $(i_1,i_2)\in I^2$, we have $|M_{i_1}\cap M_{i_2}|\geq f_2(\rho)|X|$, as required.

\item Case 2: $|X|\geq 4\lceil 2\rho^{-1}\rceil\rho^{-1}$. Set $\epsilon:=\rho/2$; we will be using the fact (see \cite[Lemma 2.1.2(3)]{Bor16a} and note that the necessary inequality $|I'|\geq 2\lceil\epsilon^{-1}\rceil=2\lceil2\rho^{-1}\rceil$ holds by our case assumption on $|X|$ and the assumption on $|I'|$ below) that for all subsets $I'\subseteq I$ such that $|I'|\geq\epsilon|X|=\frac{\rho}{2}|X|$, there is an $i'\in I'$ such that for at least $\frac{1}{2\lceil\epsilon^{-1}\rceil}|I'|\geq\frac{\rho}{4\lceil2\rho^{-1}\rceil}|X|$ many $j'\in I'\setminus\{i'\}$, we have

\[
|M_{i'}\cap M_{j'}|\geq\frac{2\epsilon}{\lceil\epsilon^{-1}\rceil\cdot(\lceil\epsilon^{-1}\rceil+1)}|X|=f_2(\rho)|X|.
\]

Set $I_0:=I$. Assume that we have already defined, for some $t\leq\frac{\rho}{3}|X|+1$, a strictly decreasing chain of index sets $I_0\supset I_1\supset\cdots\supset I_t$ and indices $i_0,\ldots,i_{t-1}\in I$ such that for $k=1,\ldots,t$, $I_k=I_{k-1}\setminus\{i_{k-1}\}$, and such that for $k=0,\ldots,t-1$, there are at least $\frac{\rho}{4\lceil 2\rho^{-1}\rceil}|X|$ many $j\in I_k$ with $|M_{i_{k-1}}\cap M_j|\geq f_2(\rho)|X|$. Then since

\[
|I_t|=|I|-t\geq\rho|X|-(\frac{\rho}{3}|X|+1)\geq\rho|X|-\frac{\rho}{2}|X|=\epsilon|X|;
\]

we conclude that there is also $i_t\in I_t$ such that for at least $\frac{\rho}{4\lceil 2\rho^{-1}\rceil}|X|$ many $j\in I_t\setminus\{i_t\}$, we have $|M_{i_t}\cap M_j|\geq f_2(\rho)|X|$. Hence we can set $I_{t+1}:=I_t\setminus\{i_t\}$ to proceed with the recursion.

Altogether, this yields at least

\[
\frac{\rho}{3}|X|\cdot\frac{\rho}{4\lceil 2\rho^{-1}\rceil}|X|=\frac{\rho^2}{12\lceil 2\rho^{-1}\rceil}|X|^2\geq f_1(\rho)|X|^2
\]

many pairs $(i_1,i_2)\in I^2$ such that $|M_{i_1}\cap M_{i_2}|\geq f_2(\rho)|X|$ in this case as well.
\end{itemize}
\end{proof}

\begin{proof}[Proof of Theorem \ref{mainTheo}]
Fix a homomorphism $\varphi:G^d\rightarrow G$ such that $S:=\{\vec{g}\in G^d\mid\varphi(\vec{g})=w_G(\vec{g})\}$ has size at least $\rho|G|^d$. We apply Lemma \ref{setLem} to the family $(\vec{s}S)_{\vec{s}\in S}$ of subsets of $X:=G^d$ to conclude that there are at least $f_1(\rho)|G|^{2d}$ many pairs $(\vec{s},\vec{t})\in S^2$ such that

\[
|S\cap \vec{s}^{-1}\vec{t}S|=|\vec{s}S\cap\vec{t}S|\geq f_2(\rho)|G|^d.
\]

Hence for at least $f_1(\rho)f_2(\rho)|G|^{3d}$ many triples $(\vec{s},\vec{t},\vec{u})\in S^3$, we have $\vec{s}^{-1}\vec{t}\vec{u}\in S$, and so, writing $\vec{a}=(a_1,\ldots,a_d)$ for $a\in\{s,t,u\}$,

\begin{align*}
w_G(s_1^{-1}t_1u_1,\ldots,s_d^{-1}t_du_d) &=w_G(\vec{s}^{-1}\vec{t}\vec{u})=\varphi(\vec{s})^{-1}\varphi(\vec{t})\varphi(\vec{u})=w_G(\vec{s})^{-1}w_G(\vec{t})w_G(\vec{u}) \\
&=w_G(s_1,\ldots,s_d)^{-1}w_G(t_1,\ldots,t_d)w_G(u_1,\ldots,u_d).
\end{align*}
\end{proof}

\section{Applications of Theorem \ref{mainTheo}}

We give two applications of Theorem \ref{mainTheo} in this note. The first is Corollary \ref{mainCor}, the proof of which is easy now:

\begin{proof}[Proof of Corollary \ref{mainCor}]
By our assumption that $w$ is reduced of length at least $2$, the word equation $w(x_1^{-1}y_1z_1,\ldots,x_d^{-1}y_dz_d)=w(x_1,\ldots,x_d)^{-1}w(y_1,\ldots,y_d)w(z_1,\ldots,z_d)$ is nontrivial (when moving the symbols from the right-hand side to the left-hand side by corresponding right multiplications, the last factor $(x_i^{-1}y_iz_i)^{\pm1}$ on the left-hand side does not fully cancel). Hence the assertion follows from \cite[Theorem 1.1]{LS17a} (more precisely, from the fact, implied by \cite[Theorem 1.1]{LS17a}, that the orders of the nonabelian composition factors of a finite group $G$ such that, for a given nontrivial reduced word $v\in\F(X_1,\ldots,X_t)$, the word map $v_G$ has a fiber of size at least $\rho|G|^t$, are bounded in terms of $v$ and $\rho$).
\end{proof}

The second application illustrates the use of Theorem \ref{mainTheo} for studying a concrete example where one can say more about the structure of $G$ than Corollary \ref{mainCor} does. The \emph{commuting probability} of a finite group $G$ is the probability that two independently uniformly randomly chosen elements of $G$ commute (in other words, it is the number $|\{(g,h)\in G^2\mid gh=hg\}|/|G|^2$):

\begin{corollary}\label{sideCor}
Let $\rho\in\left(0,1\right]$. A finite group $G$ whose group multiplication, viewed as a function $G^2\rightarrow G$, agrees with some homomorphism $G^2\rightarrow G$ on at least $\rho|G|^2$ many arguments has commuting probability at least $\frac{f_1(\rho)f_2(\rho)}{2-f_1(\rho)f_2(\rho)}$.
\end{corollary}

\begin{proof}
Set $\epsilon:=f_1(\rho)f_2(\rho)$. We make a case distinction.

\begin{itemize}
\item Case 1: $|G|<\sqrt[4]{2/\epsilon}$. Then since each element of $G$ commutes with itself, the commuting probability of $G$ is at least $|G|^{-1}>\sqrt[4]{\epsilon/2}$, which is bounded from below by $\epsilon/(2-\epsilon)$ since by the definitions of $f_1$ and $f_2$, we certainly have $\epsilon<\sqrt[3]{1/2}$.

\item Case 2: $|G|\geq\sqrt[4]{2/\epsilon}$. By our proof of Theorem \ref{mainTheo}, applied with $w:=X_1X_2$, we conclude that for at least $\epsilon|G|^6$ many sextuples $(s_1,s_2,t_1,t_2,u_1,u_2)\in G^6$, the equation

\[
s_1^{-1}t_1u_1s_2^{-1}t_2u_2=s_2^{-1}s_1^{-1}t_1t_2u_1u_2
\]

holds, which is equivalent to

\begin{equation}\label{eq1}
s_1s_2s_1^{-1}=t_1t_2u_1t_2^{-1}s_2u_1^{-1}t_1^{-1}.
\end{equation}

We claim that there are at least $\frac{\epsilon}{2}|G|^4$ many quadruples $(t_1,t_2,u_1,u_2)\in G^4$ each having the property that for at least $\frac{\epsilon}{2-\epsilon}|G|^2$ many pairs $(s_1,s_2)\in G^2$, Equation (\ref{eq1}) holds. Indeed, otherwise, we would get that the number of sextuples satisfying Equation (\ref{eq1}) is strictly smaller than $(\frac{\epsilon}{2}+(1-\frac{\epsilon}{2})\cdot\frac{\epsilon}{2-\epsilon})|G|^6=\epsilon|G|^6$, a contradiction. Now by our case assumption, we have that $\frac{\epsilon}{2}|G|^4\geq 1$, so we can fix a quadruple $(t_1,t_2,u_1,u_2)$ with the described property, and then the set $M$ consisting of the at least $\frac{\epsilon}{2-\epsilon}|G|^2$ matching pairs $(s_1,s_2)$ has the following property: For each $s_2\in G$, the set $N_{s_2}$ of all $s_1\in G$ such that $(s_1,s_2)\in M$ is contained in a single coset of the centralizer $\C_G(s_2)$, so $|N_{s_2}|\leq|\C_G(s_2)|$. From this, we infer that

\[
|\{(g,h)\in G^2\mid gh=hg\}|=\sum_{h\in G}{|\C_G(h)|}\geq\sum_{h\in G}{|N_h|}=|M|\geq\frac{\epsilon}{2-\epsilon}|G|^2,
\]

as required.
\end{itemize}
\end{proof}

\end{document}